\documentclass[11pt]{amsart}

\usepackage[utf8x]{inputenc} 
\usepackage[a4paper]{geometry} 
\usepackage{enumerate}  
\usepackage{amsmath} 
\usepackage{amssymb}
\usepackage{color}
\usepackage{multirow}
\usepackage{hyperref}  
\usepackage{verbatim}
                 \hypersetup{ pdfborder={0 0 0}, 
                              colorlinks=true, 
                              linktoc=page,
                              pdfauthor={Michael Hoff and Andreas Leopold Knutsen}, 
                              pdftitle={Brill--Noether general K3 surfaces with the maximal number of elliptic pencils of minimal degree}} 

\definecolor{Gray}{gray}{0.9}                            

\usepackage[arrow,curve,matrix,tips,frame]{xy}

\theoremstyle{definition}
\newtheorem{definition}{Definition}[section]

\newtheorem{remark}[definition]{Remark}

\theoremstyle{plain}
\newtheorem{theorem}[definition]{Theorem}

\newtheorem{corollary}[definition]{Corollary}

\newtheorem{lemma}[definition]{Lemma}

\newtheorem{proposition}[definition]{Proposition}

\newtheorem*{claim*}{Claim}

\theoremstyle{remark}
\newtheorem{example}[definition]{Example}

\newcommand{\OO}{{\mathcal{O}}}

\newcommand{\CC}{{\mathbb{C}}}      
\newcommand{\PP}{{\mathbb{P}}}          

\newcommand{\cF}{{\mathcal{F}}}
\newcommand{\cP}{{\mathcal{P}}}
\newcommand{\cM}{{\mathcal{M}}}
\newcommand{\cW}{{\mathcal{W}}}
\newcommand{\cH}{{\mathcal{H}}}
\newcommand{\cI}{{\mathcal{I}}}
\newcommand{\cE}{{\mathcal{E}}}
\newcommand{\T}{{\mathcal{T}}}
\newcommand{\hh}{{\mathfrak{h}}}
\newcommand{\fL}{{\mathfrak{L}}}
\newcommand{\fM}{{\mathfrak{M}}}
\newcommand{\fN}{{\mathfrak{N}}}
\newcommand{\fU}{{\mathfrak{U}}}

\providecommand{\rk}{\mathop{\rm rk}} 
\providecommand{\Pic}{\mathop{\rm Pic}} 
\def\cork{\operatorname{cork}}

\numberwithin{equation}{section}

\title[K3 surfaces with many elliptic pencils]{{Brill--Noether general K3 surfaces with the maximal number of elliptic pencils of minimal degree}}

\author[M. Hoff]{Michael Hoff} 
\address{Universit\"at des Saarlandes, Campus E2 4, D-66123 Saarbr\"ucken, Germany}
\email{\href{mailto:hahn@math.uni-sb.de}{hahn@math.uni-sb.de}} 

\author[ A. L.~Knutsen]{Andreas Leopold Knutsen} 
\address{University of Bergen, Department of Mathematics, N-5020 Bergen, Norway}
\email{\href{mailto:andreas.knutsen@math.uib.no}{andreas.knutsen@math.uib.no}} 

\date{\today} 

\begin{document}


\begin{abstract}
We explicitly construct Brill--Noether general $K3$ surfaces of genus $4,6$ and $8$ having the maximal number of elliptic pencils of degrees $3, 4$ and $5$, respectively, and study their moduli spaces and moduli maps to the moduli space of curves. 
As an application we prove the existence of Brill--Noether general $K3$ surfaces of genus $4$ and $6$ without stable Lazarsfeld--Mukai bundles of minimal $c_2$.
\end{abstract}

\maketitle

\section{Introduction}

It is well known that a general curve of genus $g\le 9$ or $g=11$ can be realized as a linear section of a primitively polarized $K3$ surface, cf. \cite{Muk88,Mukg11}. Since for even $g$ a general curve $C$ carries a finite number of pencils of minimal degree $\frac{g}{2}+1$, it is natural to ask whether one can simultaneously extend $C$ and all or some of these pencils to some $K3$ surfaces for $g=4,6,8$. 
This question is connected to the existence of non-stable Lazarsfeld--Mukai bundles. Indeed, the Lazarsfeld--Mukai bundle associated to  a pencil on a smooth curve on the $K3$ surface induced by an elliptic pencil on the surface is necessarily not stable, cf. Lemma \ref{LMnotstable}. 

Using vector bundle methods, Mukai \cite{Muk02} showed that the projective model of any \emph{Brill--Noether general} $K3$ surface $(S,L)$ is obtained as sections of homogeneous varieties for $g \in \{6,\dots, 10, 12\}$. By definition, cf. \cite[Def. 3.8]{Muk02}, a polarized $K3$ surface $(S,L)$ of genus $g$ is Brill--Noether general if $h^0(M) h^0(N) < g + 1 = h^0(L)$ for any non-trivial decomposition $L \sim M+N$. In these low genera this is equivalent to all the smooth curves in the linear system $|L|$ being Brill--Noether general, due to techniques in \cite{L86, GL87} (see \cite[Lemma 1.7]{GLT15}). 
Using Mukai's results, we will study projective models of Brill--Noether general $K3$ surfaces of genus $g\in \{4,6,8\}$ containing the maximal possible number of elliptic pencils of degree $\frac{g}{2} + 1$. 

The goal of our paper is threefold: 
\begin{enumerate}
 \item We provide explicit constructions/equations of $K3$ surfaces with special geometric features. 
 \item We describe their moduli spaces as lattice polarized $K3$ surfaces and the corresponding moduli map to the moduli space of curves of genus $g$. 
 \item We study the slope-stability of Lazarsfeld--Mukai bundles of hyperplane sections on such $K3$ surfaces.
\end{enumerate}

Our main results are the following. 
\\
$\bullet$ \S \ref{genus4}: We prove that a general curve $C$ of genus $4$ is a linear section of a smooth $K3$ surface $S$ such that its two $g^1_3$s (which are well-known to be auto-residual) are induced by two elliptic pencils $|E_1|$ and $|E_2|$ on $S$ satisfying $C\sim E_1+ E_2$, cf. Proposition \ref{dybala4}. Furthermore, the moduli space parametrizing such $K3$ surfaces is unirational (and $18$-dimensional), cf. Proposition \ref{higuain4}. We believe that these results should be known, but could not find any reference. 
\\
$\bullet$ \S \ref{genus6}:
A general curve $C$ of genus $6$ carries precisely five pencils $|A_1|,\dots, |A_5|$ of minimal degree $4$ which satisfy $2K_C\sim A_1+ \dots + A_5$ (see \cite[p. 209ff]{ACGH}). We prove that $C$ is a linear section of a smooth $K3$ surface $S$ such that its five $g^1_4$s are induced by five elliptic pencils $|E_1|,\dots,|E_5|$ on $S$ satisfying $2 C \sim E_1+\dots+E_5$, cf Theorem \ref{ModuliMapDominant}(a).  We prove that the moduli space parametrizing such pairs $(S,C)$ is unirational, cf. Theorem \ref{unirationality}(b). The moduli space of the underlying $K3$ surfaces was already studied in \cite{AK11} where it was shown to be birational to the moduli space $\cM_6$ of curves of genus $6$ (whence rational, cf. \cite{SB89}). Our approach shows that this moduli space is exactly the locus of Brill--Noether general $K3$ surfaces that cannot be realized as quadratic sections of a smooth quintic Del Pezzo threefold (but as quadratic sections of a cone over a smooth quintic Del Pezzo surface), cf. Remark \ref{totti6}(b). 
\\
$\bullet$ \S \ref{genus8}:
A general curve $C$ of genus $8$ carries precisely $14$ pencils of degree $5$. An easy lattice computation shows that at most $9$ can be extended to a $K3$ surface containing $C$. We prove that this bound is reached in codimension $3$ in the moduli space $\cM_8$, and for a general curve only six out of its $14$ pencils can be extended to elliptic pencils on a $K3$ surface, cf. Corollary \ref{dybala8}. 
 We prove that the moduli spaces of such $K3$ surfaces containing $i$ elliptic pencils are unirational for $1\le i\le 6$ and $i=9$, cf. Theorems \ref{thm:uniN9} and \ref{thm:uniN6}. 
\\
$\bullet$ \S \ref{LMbundles}: The $K3$ surfaces constructed in Section \ref{genus4} (respectively \ref{genus6}) provide examples of $K3$ surfaces without stable (resp. semistable) Lazarsfeld--Mukai bundles with $c_2 = 3$ (resp. $4$), cf. Corollary \ref{voeller4} (resp. \ref{voeller6}). This shows in particular the sharpness of a result  of Lelli-Chiesa \cite[Thm. 4.3]{LC13},  cf. Remark \ref{RemarkLC}.

\subsection*{Notation and conventions} 

We work over $\mathbb{C}$. We will denote $V_n$ an $n$-dimensional vector space and $G(k,V_n)$ (respectively $G(V_n,k)$) the Grassmannian of $k$-dimensional sub- (resp. quotient-) spaces of $V_n$. The projective space of one-dimensional sub- (resp. quotient-) spaces is denoted $\PP_*(V_n)$ (resp. $\PP^*(V_n))$.

\subsection*{Acknowledgements}

The authors benefitted from conversations with Christian Bopp and Frank-Olaf Schreyer and acknowledge support from grant n. 261756 of the Research Council of Norway. 

\section{Lattice polarized $K3$ surfaces and their moduli spaces}

Let $\hh$ be a lattice. The moduli space $\cF^\hh$ of $\hh$-polarized $K3$ surfaces parametrizes pairs $(S,\varphi)$ (up to isomorphism) consisting of a $K3$ surface $S$ and a primitive lattice embedding $\varphi: \hh \to \Pic(S)$ such that  $\varphi(\hh)$ contains an ample class. It is a quasi-projective irreducible $(20 - \rk(\hh))$-dimensional variety by \cite{Do96}.

If $(S,\varphi)\in \cF^\hh$ is an $\hh$-polarized $K3$ surface and $L\in \hh \cong\varphi(\hh)$ is a distinguished class with $L^2=2g-2\ge 2$, 
one may consider the open subset
$$
\cF^\hh_g=\left\{ (S,\varphi)\ \big| \  (S,\varphi) \in \cF^\hh \text{ and } L \text{ ample } \right\}
$$
of the moduli space $\cF^\hh$, which may also be considered as a subset of the moduli space $\cF_g$ of polarized $K3$ surfaces of genus $g$.  
Furthermore, let  
$\cP^\hh_g$
denote the moduli space of triples $(S,\varphi,C)$ where $C\in |L|$ is a smooth irreducible curve in the distinguished linear system.
Then we have moduli maps
$$
m_g:\cP^\hh_g \to \cM_{g}.
$$
 Since in our cases of study  it will be clear what the distinguished class $L$ will be, we will often skip the index $g$ in $\cF_g^\hh$ and $\cP_g^\hh$. 

\section{$K3$ surfaces of genus $4$}\label{genus4}

We will show the unirationality of the moduli space $\cF^{\fU(3)}$ of lattice polarized $K3$ surfaces where $\fU$ is the hyperbolic lattice of rank $2$. We believe that this result should be well-known, but we could not find any reference. 

The following example is well-known, but we include it for the sake of the reader and it serves as an introduction for our next results and constructions.  

\begin{example}[The moduli space of $K3$ surfaces of genus $4$] \label{exK3g4}
A smooth polarized $K3$ surface $S\subset \PP^4$ of genus $4$ is the complete intersection of a quadric $Q$ and a cubic hypersurface $Y$ in $\PP^4$. The quadric $Q=V(q)$ and the cubic $Y=V(y)$ are given by polynomials $q \in H^0(\PP^4, \OO_{\PP^4}(2))$ and $y \in H^0(\PP^4, \OO_{\PP^4}(3))$ of degrees $2$ and $3$, respectively. 
 
The moduli space $\cF_4$ of $K3$ surfaces of genus $4$ is described as follows. The quadric has to be of rank at least $4$ since otherwise $S$ will be singular. Let $V\subset H^0(\PP^4, \OO_{\PP^4}(2))$ be the open subset consisting of quadratic equations of rank $\geq 4$. For a chosen equation $q$ we need to pick a cubic $y$ such that $y$ is no multiple of $q$, and the intersection of $Q$ and $Y$ should be smooth. Let $V_q$ be the five-codimensional quotient of $H^0(\PP^4, \OO_{\PP^4}(3))$ parametrizing non-multiples of $q$. The desired cubic equations are parametrized by an open subset $W_q\subset V_q$. Let $W$ be the iterated Grassmannian 
$$
\xymatrix{
W \ar[rr]^{G(1,W_q) \ \ \ \ } & & \PP_*(V) \cong \PP^{14}
}
$$
whose fibers are Grassmannians of one-dimensional subspaces of $W_q$.  Then $\cF_4$ is birational to $W$ modulo the automorphism group of $\PP^4$, whence  $\cF_4$ is unirational.  Note further that a dimension count yields 
$$\dim V + \dim W_q - \dim PGL(5) = \left(\binom{6}{2}-1\right) + \left(\binom{7}{3} - 1 - 5\right) - (5^2 -1) = 19,$$ 
as expected. 

\end{example}

\subsection{K3 surfaces of genus 4 with an elliptic pencil of degree 3} \label{exK3g4k3} 
With notation as in the previous example let $S\subset \PP^4$ be a smooth $K3$ surface of genus $4$ with polarization $L=\OO_S(1)$.  Assume   that there exists a class $E\in \Pic(S)$ such that $E^2=0$ and $E.L=3$. By Riemann--Roch, $h^0(S,E)=2$ and $E'$ is a smooth elliptic normal curve for  general $E'\in |E|$.  Hence we get a pencil of elliptic normal curves. The pencil induces a rational normal scroll 
$$
X=\bigcup_{E'\in |E|} \overline{E'} \subset \PP^4
$$
of dimension $3$ and degree $2$ where $\overline{E'}=\PP^2$ is the linear span of $E'$. Thus the scroll $X$ is the unique quadric hypersurface containing $S$. Furthermore,  the scroll $X$ is singular in a point (since any two different projective planes in $\PP^4$ intersect and $X$ cannot be singular along a line), that is, $X$ is a rank $4$ quadric. 
 
We remark that the residual class $L-E$ is a second elliptic pencil of degree $3$ on $S$ and the maximal number of such pencils is two since $S\subset \PP^4$ is generated by a unique quadric. We get a $K3$ surface whose Picard lattice contains the intersection matrix with respect to the ordered basis $\{L,E\}$ (respectively $\{L-E,E\}$)
$$
\begin{pmatrix}
 6 & 3 \\ 
 3 & 0 
\end{pmatrix} 
\left(\mbox{resp. }
\begin{pmatrix}
 0 & 3 \\ 
 3 & 0 
\end{pmatrix}
= \fU(3)\right)
$$
where $\fU$ is the hyperbolic lattice of rank $2$ and $L$ is the sum of the two basis elements of square $0$. In general $\Pic (S)\cong \fU(3)$ (such $K3$ surfaces exist by
\cite[Thm. 2.9(i)]{Mor84} or \cite{Nik80}),
in which case $L$ is the unique element (up to sign) of square $6$, whence genus $4$, which is easily seen to be very ample by the classical results of Saint-Donat \cite{SD}. Furthermore, such a $K3$ surface $(S,L)$ is Brill--Noether general. 

Recall from the introduction that $\cF^{\fU(3)}$ is the moduli space of $\fU(3)$-polarized $K3$ surfaces. 

\begin{proposition}\label{higuain4}
 The moduli space $\cF^{\fU(3)}$ is unirational.  
\end{proposition}

\begin{proof}
By what we said, a general element in $\cF^{\fU(3)}$ comes equipped with a unique embedding into $\PP^4$ (up to the action of the projective linear group), as a complete intersection of a cubic and a rank $4$ quadric, singular in a point. The converse holds true: if a smooth surface $S\subset \PP^4$ is a complete intersection of a rank $4$ quadric hypersurface $Q$ and a cubic hypersurface, then the two rulings on $Q$ cut out two residual elliptic pencils of degree $3$ on $S$. 
 
We describe a birational model of the moduli space $\cF^{\fU(3)}$ by modifying the construction in Example \ref{exK3g4}, keeping the notation therein. 

Let $V'\subset H^0(\PP^4, \OO_{\PP^4}(2))$ be the subset of quadratic equations of rank $4$. Since a rank $4$ quadric is a cone over a smooth quadric in $\PP^3$, the space $V'$ is isomorphic to an open subset of a $\PP^4$-bundle over $\PP H^0(\PP^3, \OO_{\PP^3}(2))$ and is therefore unirational. Pick $q\in V'$. Then the moduli space $\cF^{\fU(3)}$ is birational to the iterated Grassmannian 
$$
\xymatrix{
W' \ar[rr]^{G(1,W_q) \ \ } & & V'
}
$$ modulo automorphisms and is therefore unirational, too. (Since $\dim V' = \binom{5}{2}-1+4 = 13$, a dimension count yields that $\cF^{\fU(3)}$ is a codimension one subspace of $\cF_4$, as expected.)
\end{proof}

\begin{remark}
Let $\fU$ be the hyperbolic lattice of rank $2$. Even if the example above should be classically known, we only found in the literature unirationality results of $\cF^{\fU(n)}$ for $n=1$ and $2$  (cf. \cite{BHK16}). Elliptic surfaces are parametrized by $\cF^{\fU}$ and double covers of $\PP^1\times \PP^1$ branched along a curve of bidegree $(4,4)$ are parametrized by $\cF^{\fU(2)}$.
\end{remark}

Recall from the introduction that $\cP^{\fU(3)}$ is the moduli space of triples $(S,\varphi,C)$ where $(S,\varphi)\in \cF^{\fU(3)}$ and $C\in |L|$ is a smooth curve of genus $4$ in the distinguished linear system.
Also recall that a  general  curve of genus $4$ has exactly two distinct  $g^1_3$s, which are auto-residual. 

\begin{proposition}\label{dybala4}
  The moduli map $\cP^{\fU(3)}\to \cM_4$ is dominant. In particular, a general curve $C$ of genus $4$ is a linear section of a smooth $K3$ surface $S$ such that its two $g^1_3$s are induced by two elliptic pencils $|E_1|$ and $|E_2|$ on $S$ satisfying $C\sim E_1+ E_2$.
\end{proposition}

\begin{proof}
 We consider a general curve $C\subset \PP^3$ of genus $4$, canonically embedded into $\PP^3$, which is a complete intersection of a smooth quadric $Q'$ and a cubic $Y'$ (the quadric $Q'$ is smooth since the two $g^1_3$s are distinct). We will construct a $K3$ surface $S\in \cF^{\fU(3)}$ with the curve $C$ as a linear section. Therefore, we choose a $\PP^4$  containing the ambient space $\PP^3$ of the curve. Let $Q\subset \PP^4$ be a cone over the quadric $Q'\subset \PP^3$, that is, a rank $4$ quadric whose hyperplane section with the given $\PP^3$ is $Q'$. Let $Y\subset \PP^4$ be any cubic hypersurface such that $Y\cap \PP^3=Y'$. The surface $S\subset \PP^4$ can be chosen as the complete intersection of $Q$ and $Y$. Then, the pair $(S,C)$ is an element of $\cP^{\fU(3)}$ by construction, and the dominance of the moduli map follows. 
 The last statement is immediate.
\end{proof}

\begin{remark}
 Similarly in \cite{Kon02} it is shown that the moduli space of $K3$ surfaces admitting a special automorphism of order $3$ is birational to the moduli space of curves of genus $4$ (see also \cite{AS08} for its generalization). 
\end{remark}

\section{K3 surfaces of genus 6}\label{genus6}

Inspired by the seminal work of Mukai \cite{Muk93}, we will construct a Brill--Noether general $K3$ surface $S$ of genus $6$ where every complete pencil of degree $4$ on a hyperplane section of $S$ is induced by an elliptic pencil on $S$.  Furthermore, we show that the moduli space of such lattice polarized $K3$ surfaces is unirational.

We briefly recall Mukai's construction. Let $(S,L)$ be a Brill--Noether general $K3$ surface of genus $6$. There exists a unique stable (rigid) vector bundle $\cE$ of rank $2$ on $S$ with $c_1(\cE)=L$, $h^0(S,\cE)=5$ and $h^i(S,\cE)=0$ for $i=1,2$ \cite[Prop. 5.2.7]{IP99}. This bundle induces an embedding of $S$ into the Grassmannian $G(V_5,2)$, where $V_5 = H^0(S,\cE)$, by sending $s\in S$ to the fiber $\cE_s = \cE \otimes \OO_s$. 
As described in \cite{Muk93}, a Brill--Noether general $K3$ surface $S$ is the intersection of a linear section of codimension 3 (or 4) and a quadratic section of either the Pl\"ucker embedding $G(V_5,2)\subset \PP^9$ or of its cone $\widehat{G(V_5,2)}\subset \PP^{10}$, respectively. 
 
In order to get an elliptic pencil of degree $4$ on a $K3$ surface, we need special sections of the following form. If the linear section of codimension $3$ cuts a sub-Grassmannian of type $G(4,2)$ in a quadric surface, we get an elliptic normal curve of degree $4$ on $S$ as the intersection of this quadric surface with the quadric section. A pencil of Grassmannians of type $G(4,2)$ induces a pencil of elliptic curves on $S$ and can be controlled in the dual space in the following way.  

\begin{lemma}\label{schubertVarietyDualGrassmannian}
A hyperplane corresponds to a point in the dual Grassmannian $G(2,V_5) \subset{\PP^9}^\vee$ if and only if it cuts out a Schubert subvariety. Moreover, the Schubert variety is a one-dimensional union of Grassmannians of type $G(4,2)$ contained in $G(V_5,2)$. 
\end{lemma}

We will prove the same statement for the Grassmannian $G(V_6,2)$ in  the next section (cf. Lemma \ref{mukaione})  and leave this proof to the readers. Note that two Grassmannians of type $G(4,2)$ in $G(V_5,2)$ intersect in a $2$-plane. Hence, two elliptic curves of distinct pencils of degree $4$  with respect to  $L$ intersect in two points. 
This can also be seen in the following way: if $E_1$ and $E_2$ are such  elliptic curves, then $E_1.E_2\ge2$ (as each $|E_i|$ is a pencil); moreover,  since $(L-E_1)^2=2$,  one also has  $4 - E_1.E_2 = E_2.(L-E_1) \ge 2 $, whence  $E_1.E_2\le 2$. Also inspired by the previous example of $K3$ surfaces of genus $4$, we will construct a $K3$ surface with Picard lattice of the following form: 

$$
\begin{pmatrix}
 10 & 4 & 4 & \dots & 4 \\ 
 4 & 0 & 2 & \dots & 2 \\
 4 & 2 & 0  & \ddots & \vdots \\
 \vdots & \vdots & \ddots & \ddots & 2 \\
 4 & 2 & \dots & 2 & 0
\end{pmatrix}
$$

An easy computation shows that  the rank can be at most five (otherwise the matrix has at least two non-negative eigenvalues).  
Let $\fM$ be the lattice given by the following intersection matrix 

$$
\fM=
\begin{pmatrix}
 10 & 4 & 4 & 4 & 4 \\ 
 4 & 0 & 2 & 2 & 2 \\
 4 & 2 & 0 & 2 & 2 \\
 4 & 2 & 2 & 0 & 2 \\
 4 & 2 & 2 & 2 & 0
\end{pmatrix}.
$$

We denote $S$ a $K3$ surface with the above Picard lattice $\fM$ of rank $5$  (which exists by \cite[Thm. 2.9(i)]{Mor84} or \cite{Nik80})  and let $L$ be the basis element of square $10$. 
Let $E_i$, $i=1,\dots,4$,  be the generators of square zero.  Note that $E_5:= 2L-E_1-E_2-E_3-E_4$ is also an  element of square zero and degree $4$ with respect to $L$.

The lattice $\fM$ is also generated by elements $s_0,s_1, \dots, s_4$ where $s_0 = E_1+\cdots + E_4 - L$ and $s_i = s_0-E_i$, $i=1,\dots,4$, with intersection matrix

 $$
\begin{pmatrix}
 2 & 0 & 0 & 0 & 0 \\ 
 0 & -2 & 0 & 0 & 0 \\
 0 & 0 & -2 & 0 & 0 \\
 0 & 0 & 0 & -2 & 0 \\
 0 & 0 & 0 & 0 & -2
\end{pmatrix}.
$$

(This is the lattice considered in \cite{AK11}.)  We may assume that $s_0$ is big and nef by standard arguments (see  \cite[VIII, Prop. 3.10]{BHPV}). Note that $L = 3s_0 - \sum_{i=1}^4 s_i$, $E_i = s_0-s_i$ for $i=1,\dots,4$ and $E_5 = 6s_0 - 3 \sum_{i=1}^4 s_i$.

\begin{lemma}\label{cr6}
\begin{enumerate}
 \item [(a)] The class $L$ is ample. 
 \item [(b)] The $K3$ surface $(S,L)$ is Brill--Noether general. 
 \item [(c)] The classes $E_1,\dots, E_5$ define elliptic pencils and are the only classes in $\Pic(S)$ of square $0$ and degree $4$ with respect to $L$.
\end{enumerate}

\end{lemma}

\begin{proof}
Let $\Delta = \sum_{i=0}^4 a_i s_i$ be an arbitrary class. Then $\Delta^2 = 2 a_0 - 2\sum_{i=1}^4 a_i$, whence 
$L.\Delta = 8a_0 - \Delta^2$.  If $\Delta$ is effective, then $a_0 = \frac{1}{2}s_0.\Delta  \ge 0$ since $s_0$ is nef.
It follows that $L.\Delta \ge 2$ for any $(-2)$-curve $\Delta$, whence (a) is  proved. 
It also immediately follows that there exists no nontrivial effective class $\Delta$ such that either $\Delta^2=0$ and $\Delta.L \le 3$ or $\Delta^2=2$ and $\Delta.L = 5$. This implies (b) by either a direct computation using the definition of Brill--Noether generality or invoking, e.g., \cite[Prop. 10.5]{JK04} and \cite{SD}, or \cite[Lemma 1.7]{GLT15}.
 
To prove that $|E_i|$ is an elliptic pencil, it suffices to show that $E_i$ is nef by \cite{SD}. If $E_i$ for some $i\in \{1,\dots, 5\}$ is not nef, 
there exists a $(-2)$-curve $\Gamma$ with $\Gamma.E_i\le 0$. 
Let $k:=-\Gamma.E_i \ge 1$. Then $(E_i-k\Gamma)^2 = 0$ and $E_i-k\Gamma$ is effective and nontrivial with $(E_i-k\Gamma).L \le 4-k\le 3$ by ampleness of $L$, a contradiction to the Brill--Noether generality.  Finally, if $F$ is another effective class with $F^2=0$, then $F.E_i\ge 2$ for all $i$, since $F$ moves in (at least) a pencil. Thus $F.L = \frac{1}{2}F.(E_1+\dots+E_5)\ge 5$.
\end{proof}
 
We will show that the general curve lies on a six-dimensional family of such $K3$ surfaces of Picard rank $5$. We will use the cone over the Grassmannian $G(V_5,2)$ in $\PP^{10}$. 

\subsection{K3 sections of a cone of the Grassmannian $G(V_5,2)$}

 Let $\fM$ be the rank $5$  lattice above.  
Let $\cF^{\fM}$ be the moduli space of $\fM$-polarized $K3$ surfaces  and  $\cP^{\fM}$ be as in the introduction. Recall that $\dim \cF^{\fM}=15$ and $\dim \cP^{\fM} = 21$.
 Also recall that a general genus $6$ curve carries precisely five elliptic pencils $|A_1|,\ldots,|A_5|$ of degree four, which satisfy $2K_C \sim A_1+\cdots+A_5$.

By \cite{AK11} the moduli space $\cF^{\fM}$ is birational to $\cM_6$, which is well-known to be rational by \cite{SB89}. More precisely, Artebani and Kond\={o} show that $\cF^{\fM}$ is the locus of $K3$ surfaces admitting a double cover to a quintic Del Pezzo surface branched along a curve of genus $6$. In particular, this shows that the moduli map 
$\psi: \cP^{\fM}\to \cM_6$ is dominant since we get a section. However, the pairs $(S,L)$ admit automorphisms fixing $L$, whence $\cP^{\fM}$ is not birational to a $\PP^6$-bundle over $\cF^{\fM}$ and one cannot conclude its unirationality from the rationality of $\cF^{\fM}$. We will show by our construction that $\cP^{\fM}$ is unirational and that $\cF^{\fM}$ is the space of polarized $K3$ surfaces of genus $6$ such that all the five $g^1_4$s of their smooth curve sections are induced by elliptic pencils on the surfaces.
 
\begin{theorem} 
 (a) \label{ModuliMapDominant} The moduli map $\psi: \cP^{\fM}\to \cM_6$ is dominant. Furthermore, a general curve $C$ of genus $6$ is a linear section of a smooth $K3$ surface $S$ such that its five $g^1_4$s are induced by five elliptic pencils $|E_1|,\dots,|E_5|$ on $S$ satisfying $2 C \sim E_1+\dots+E_5$. 

(b) \label{unirationality} $\cP^{\fM}$ is unirational.
\end{theorem}

\begin{proof}
(a)   We will describe a $K3$ surface containing the general curve in $\cM_6$ as well as the geometry describing the elliptic pencils on the $K3$ surface. This is based on Mukai's result \cite[\S 6]{Muk93}.
 
 Let $C\in \cM_6$ be a general curve of genus $6$ which is given  as follows. We fix a Pl\"ucker embedding of the Grassmannian $G(V_5,2)\subset \PP^9$. Then there exists a projective $5$-space  $P\subset \PP^9$ as well as a quadric hypersurface $Q\subset P$ such that $C=P\cap Q\cap G(V_5,2)$.
 
 Let $P^\vee =\PP^3\subset {\PP^9}^\vee$ be the dual space. As $C$ is assumed to be general, $W^1_4(C)$ is finite-dimensional,  more precisely $W^1_4(C)$ consists of five smooth points, and is isomorphic to $P^\vee\cap G(2,V_5)\subset {\PP^9}^\vee$, that is, the intersection of $P^\vee$ and the dual Grassmannian $G(2,V_5) = G(V_5,2)^\vee\subset {\PP^9}^\vee$. By Lemma \ref{schubertVarietyDualGrassmannian} each point of $P^\vee\cap G(2,V_5)$ corresponds to a pencil of Grassmannians of type $G(4,2)$ in $\PP^9$. This pencil induces a cubic scroll in $\PP^9$ whose restriction to $C$ cuts out the corresponding point of $W^1_4(C)$.  
  
 Now let $\widehat{G(V_5,2)}\subset \PP^{10}$ be the cone over the Grassmannian $G(V_5,2)$ with vertex point $v$. We denote  $\widehat{G(2,V_5)}\subset {\PP^{10}}^\vee$ the cone over the dual Grassmannian with vertex $w$ such that $\widehat{G(2,V_5)}=\widehat{G(V_5,2)}^\vee$. 
 We consider the given projective $5$-space $P$ as a subspace of $\PP^{10}$. 
 
 Let $P_v=\overline{P+v}$ be the span of $P$ and the vertex $v$. Let $Q'\subset P_v$ be a quadric hypersurface such that $Q'\cap  P = Q$. We get a $K3$ surface  $S=\widehat{G(V_5,2)}\cap P_v \cap Q'$,  which we can assume to be smooth for general $Q'$.  
 Then the dual space of this $P_v$ is exactly the above $P^\vee$. As above the five intersection points $P^\vee\cap \widehat{G(V_5,2)} = P^\vee \cap G(V_5,2)$ correspond to five pencils of Grassmannians in $\PP^{10}$ whose restriction to $S$ are the five elliptic pencils of degree $4$ on $S$. We get the desired $K3$ surface with the right Picard lattice.  

(b) Recall that any canonical model of a general curve of genus $6$ can be realized as a quadratic section of a fixed quintic Del Pezzo surface $Y\subset \PP^5$ (see \cite{SB89}).
 
 We fix a $\PP^6\supset \PP^5$ and a point $v\in \PP^6$. Let $\widehat{Y}$ be the cone over $Y$ with vertex $v$. 
 For a general curve $C\in \cM_6$ we consider the linear system $\fL_C$ of quadratic sections of $\widehat{Y}$ containing $C$. We have $\dim \fL_C = h^0(\PP^6,\OO_{\PP^6}(2))-h^0(\PP^5,\OO_{\PP^5}(2))-1 = 6$. 
 We define the incidence correspondence 
 $$
 I = \{(C,S)\ |\ C\subset S\}\subset |\OO_Y(2)|\times |\OO_{\widehat{Y}}(2)| = \PP^{15} \times \PP^{22} 
 $$
 together with the projection $\pi:I\to |\OO_Y(2)|$, whose fibers are given by $\fL_C$. It follows that $\pi$ has the structure of a $\PP^6$-bundle, whence $\dim(I) = 15 + 6 = 21$.
 
 By the proof of part (a) the general member of $\fL_C$ is a smooth $K3$ surface in $\cF^{\fM}$ (note that $P=\PP^5$, $P_v=\PP^6$, $Y=P\cap G(V_5,2)$ and $\widehat{Y} = \widehat{G(V_5,2)}\cap P_v$ in the notation of that proof). Hence, we get a natural rational moduli map  $\varphi: I \dashrightarrow \cP^{\fM}$.  Since $I$ is unirational, the corollary will follow if we prove that $\varphi$ is dominant, equivalently, generically finite, since $\cP^{\fM}$ is irreducible of the same dimension as $I$.
 
 Assume therefore that $\varphi$ has  positive-dimensional  fibers. Since the rational moduli map $|\OO_Y(2)|\dashrightarrow \cM_6$ is finite, the fibers of $\varphi$ lie in fibers of $\pi$. Hence, the $K3$ surfaces in $\fL_C$ do not have maximal variation in moduli. 
 Note that $\fL_C$ contains the quadratic sections of the form $Y\cup Y'$ where $Y'\in \PP H^0(\widehat{Y},\OO_{\widehat{Y}}(1))$ which form a hypersurface in $\fL_C$. 
 Hence a general one-dimensional family in $\fL_C$ is non-isotrivial, a contradiction.
\end{proof}

\begin{remark}\label{totti6}
 (a) The proof of Corollary \ref{unirationality} shows that our construction dominates the moduli space $\cF^{\fM}$, that is, the general $K3$ surface in $\cF^{\fM}$ is a quadratic section of a cone over a quintic Del Pezzo surface in $\PP^5$. 
 
 (b) By \cite{Muk93}, all Brill--Noether general $K3$ surfaces of genus $6$ can be realized as a quadratic section of either a smooth quintic Del Pezzo threefold in $\PP^6$ or a cone over a quintic Del Pezzo surface. Item (a) shows that $\cF^{\fM}$ is precisely the locus of $K3$ surfaces that cannot be realized in a smooth Del Pezzo threefold.  
\end{remark}

\section{Lazarsfeld--Mukai bundles and their stability}\label{LMbundles}

For $K3$ surfaces constructed in Sections \ref{genus4} and \ref{genus6} we will show that these are $K3$ surfaces without any stable rank $2$ Lazarsfeld--Mukai bundle with determinant $L$ and $c_2 = 3$ or $4$, respectively. This shows in particular that the result of Lelli-Chiesa \cite[ Thm.  4.3]{LC13} about stability of rank $2$ vector bundles on $K3$ surfaces is optimal.

We recall the definition and basic properties of Lazarsfeld-Mukai bundles, which will also be needed in Section \ref{genus8}. Let $S$ be a $K3$ surface and let $C\subset S$ be a smooth curve of genus $g$ with a globally generated  line bundle $A$ of degree $d$   
with $h^0(C,A)=r+1$. The \emph{Lazarsfeld-Mukai bundle} $\cE_{C,A}$ is defined via 
an elementary transformation on $S$: 
\begin{equation} \label{eq:eltrans}
0 \longrightarrow \cE_{C,A}^\vee \longrightarrow H^0(C,A)\otimes \OO_S \longrightarrow A \longrightarrow 0,
\end{equation}
where $A$ is considered as a coherent sheaf on $S$ supported on $C$. Hence, it is a bundle of rank $r+1$ satisfying $c_1(\cE_{C,A})=[C]$, $c_2(\cE_{C,A}) = \deg A = d$ and $H^i(S,\cE_{C,A}) = 0$ for $i=1,2$. 
The bundles have been introduced by Lazarsfeld \cite{L86} and Mukai \cite{Muk89}. 
Dualizing the above sequence, we get 
$$
0 \longrightarrow H^0(C,A)^*\otimes \OO_S \longrightarrow \cE_{C,A} \longrightarrow  \omega_C\otimes A^* \longrightarrow 0,
$$
 and in particular a distinguished $(r+1)$-dimensional subspace $H^0(C,A)^* \subset H^0(\cE_{C,A})$. 
Equivalently, by \cite[Prop. 1.3]{A13}, a rank $(r+1)$-bundle $\cE$ on $S$ is a Lazarsfeld-Mukai bundle if and only if $h^1(S,\cE)=h^2(S,\cE) = 0$ and there exists  an $(r+1)$-dimensional subspace  $V\subset H^0(S,\cE)$  such that the degeneracy locus of the  evaluation morphism $V  \otimes \OO_S\to \cE$  is a smooth curve. 

\begin{lemma}\label{LMnotstable}
 If $A\in W^1_{d}(C)$ with $d\le g-1$ is induced by an elliptic pencil $|E|$ on the $K3$ surface $S$, then $\cE_{C,A}$ is not $L$-stable, where $L = \OO_S(C)$.  
\end{lemma}

\begin{proof}
 This is essentially already contained in \cite[Proof of Thm. 1.1]{AFO}. Using the snake lemma, we get the following commutative diagram 
\vspace{-0.3cm}
$$
\begin{xy}
 \xymatrix{
  & & & 0 \ar[d] & \\
  & 0 \ar[d] & 0\ar[d] & E \otimes L^* \ar[d] &  \\
 0 \ar[r] & E^* \ar[r] \ar[d] & H^0(S,E) \otimes \OO_S \ar[r] \ar[d]^{\cong} & E \ar[r] \ar[d] & 0 \\ 
 0 \ar[r] & \cE_{C,A}^{\vee} \ar[r] \ar[d] & H^0(C,A) \otimes \OO_S \ar[r] \ar[d] & A \ar[r] \ar[d] & 0 \\
  & E \otimes L^* \ar[d] & 0 & 0 & \\
  & 0 & & & 
 }
\end{xy}
$$
Dualizing the  left   column, we see that $L\otimes E^*$ is a subbundle of $\cE_{C,A}$.
Computing slopes, we get $\mu(L\otimes E^*) =  2g-2-d \ge g-1  = \mu(\cE_{C,A})$.
\end{proof}

\begin{corollary}\label{voeller4}
 Let $(S,L)\in \cF^{\fU(3)}_4$ be a Brill--Noether general polarized $K3$ surface as in Section \ref{exK3g4k3}. Then $S$ contains only $L$-strictly semistable Lazarsfeld--Mukai bundles $\cE_{C,A}$ of rank $2$ and $\det(\cE_{C,A})=L$, $c_2(\cE_{C,A}) = 3$ for $C\in |L|$ smooth. 
\end{corollary}

\begin{proof}
Note that $W^1_3(C)$ consists of exactly two residual pencils of divisors which extend to two elliptic pencils on $S$. We can apply Lemma \ref{LMnotstable}, and the corollary follows. 
\end{proof}

\begin{corollary}\label{voeller6}
 Let $(S,L)\in \cF^{\fM}_6$ be a Brill--Noether general polarized $K3$ surface as in Section \ref{genus6}. Then $S$ contains only $L$-unstable Lazarsfeld--Mukai bundles $\cE_{C,A}$ of rank $2$ and $\det(\cE_{C,A})=L$, $c_2(\cE_{C,A}) = 4$ for $C\in |L|$ smooth. 
\end{corollary}

\begin{proof}
Since $C$ is Brill--Noether general, every pencil in $W^1_4(C)$ is induced by an elliptic pencil on the $K3$ surface $S$. The result follows from Lemma \ref{LMnotstable}.
\end{proof} 

\begin{remark}\label{RemarkLC}
 Part (i) of \cite[Thm. 4.3]{LC13} implies that on any Brill--Noether general $K3$ surface $(S,L)$ of genus $g$ there are $L$-stable Lazarsfeld--Mukai bundles of determinant $L$ and $c_2$ equal to $d$ as soon as $\rho(g,1,d)>0$. (Indeed, sections of Brill--Noether general $K3$ surfaces have maximal gonality as a consequence of the definition and have Clifford dimension $1$ by ampleness of $L$, cf. \cite[Thm. 1.2]{Knutsi09} or \cite[Prop. 3.3]{CP95}). The above corollaries show that this does not always hold for $\rho(g,1,d)=0$  (at least when $g=4$ or $6$).
\end{remark}

\section{K3 surfaces of genus 8}\label{genus8}

In this section we construct $K3$ surfaces of genus $8$ with the maximal number of elliptic pencils of degree $5$. 
We recall Mukai's construction from \cite{Muk93, Muk02} and fix our notation. 

Let $(S,L)$ be a Brill--Noether general polarized $K3$ surface of genus $8$. Then there exists a unique globally generated stable vector bundle $\cE$ of rank $2$ with determinant $L$ and Euler characteristic $6$ (this can be constructed as the Lazarsfeld--Mukai bundle associated to a $g^1_5$ on any smooth $C\in |L|$ not induced by an elliptic pencil on $S$ by \cite[Prop. 1.3]{A13}). It is known that $V_6=H^0(S,\cE)$ is six-dimensional. Every fiber $\cE_s$ of $\cE$ for $s\in S$ is a $2$-dimensional quotient space of $V_6$, which induces a morphism 
$
\phi_{\cE}: S \to G(V_6,2), s \mapsto \cE_s.
$
The Grassmannian $G(V_6,2)$ is naturally embedded into $\PP^*(\bigwedge^2 V_6)=\PP^{14}$ via the Pl\"ucker embedding. The second exterior product induces a surjective map on global sections
$$
\lambda: \bigwedge^2 H^0(S,\cE) \to H^0(S,\bigwedge^2 \cE),
$$
and we get the following commutative diagram
$$
\begin{xy}
\xymatrix{
	S \ar[rr]^{\phi_{\cE}} \ar[d]_{\phi_{\bigwedge^2 \cE}}   & &  G(V_6,2) \ar[d]^{\text{Pl\"ucker}}  \\
	\PP^8=\PP^*(H^0(S,\bigwedge^2 \cE)) \ar[rr]^{\ \ \ \ \ \ \PP^*(\lambda)}       &      &   \PP^{14} 
}
\end{xy}
$$
where $\PP^*(\lambda)$ is the linear embedding induced by $\lambda$. Since $\bigwedge^2 \cE =c_1(\cE)= L$, the map $\phi_{\bigwedge^2 \cE}$ is given by the linear system $|L|$. The above diagram is cartesian, that is, $S=\PP^8 \cap G(V_6,2)$. 

Hyperplane sections of $G(V_6,2)$ are parametrized by $\PP_*(\bigwedge^2 V_6)$. The dual of $\PP^8$ is a  five-dimensional  projective space $\PP^5=\PP_*(\ker \lambda)\subset \PP_*(\bigwedge^2 V_6)$. 

 Let $C\in |L|$ be a smooth curve. The Brill--Noether generality of $(S,L)$ is equivalent to $C$ not containing a $g^2_7$ (arguing as in \cite{L86, GL87} or see \cite[Lemma 1.7]{GLT15}). Let $\cE_C$ be the restriction of $\cE$ to $C$, which is stable by \cite[\S 3]{Muk93} and $H^0(S,\cE)\cong H^0(C,\cE_C)$. As above we get a surjective morphism $\lambda_C: \bigwedge^2 H^0(C,\cE_C) \to H^0(C,\omega_C)$ 
and a commutative cartesian diagram
$$
\begin{xy}
\xymatrix{
	C \ar[rr] \ar[d]   & &  G(V_6,2) \ar[d]^{\text{Pl\"ucker}}  \\
	\PP^7 = \PP(H^0(C,\omega_C)^*) \ar[rr]^{\ \ \ \ \ \PP^*(\lambda_C)}       &      &   \PP^{14} 
}
\end{xy}
$$
since $\PP_*(\lambda_C)\cap G(2,V_6) \cong W^1_5(C)$ is finite (see \cite[Thm. C]{Muk93}). Note that $\PP_*(\lambda_C)$ is a six-dimensional space containing $\PP_*(\lambda)$.

For our purpose we state Mukai's result in the following form. 

\begin{lemma}[Mukai] \label{mukaione}
A linear intersection of $G(V_6,2)$ and $\PP^8$ is a surface (whence a Brill--Noether general $K3$ surface if smooth) if and only if the dual projective space $\PP^5$ intersects the Grassmannian $G(2,V_6)$ in the following way: for every $\PP^6\supset \PP^5$ the intersection with $G(2,V_6)\subset \PP_*(\bigwedge^2 V_6)$ is finite.
\end{lemma}

\begin{proof}
 The ''only if`` part follows from the above.  Conversely, the second condition is equivalent to any hyperplane section of the given linear section being a curve.  
\end{proof}

\subsection{Linear sections of $G(V_6,2)$ and elliptic pencils} 
We are interested in $K3$ surfaces $S\subset \PP^8$ with an elliptic pencil of minimal degree $5$. We describe a way of constructing such $K3$ surfaces. 

We use the  notation above.  Let $V_6$ be a $6$-dimensional complex vector space, and let $V_5$ be a $5$-dimensional subspace of $V_6$. We consider $G(V_5,2)\subset G(V_6,2)\subset \PP^*(\bigwedge^2 V_6)$. By a dimension count, a general $8$-dimensional linear subspace of $\PP^{14}$ intersects $G(V_5,2)$ in $5$ points. Assume instead that our $\PP^8$ intersects $G(V_6,2)$ transversally and $\PP^8\cap G(V_5,2)$ is a smooth curve, which is then an irreducible elliptic normal curve of degree $5$. Then we get a $K3$ surface $S$ with an elliptic pencil. 

\subsubsection{Dual Grassmannian and Schubert varieties}\label{secDualGrassmannian}
Even more is true. As Mukai already notices in \cite[end of p.3]{Muk93},  a hyperplane corresponds to a point in the dual Grassmannian $G(2,V_6)\subset \PP_*(\bigwedge^2 V_6)$ if and only if it cuts out a Schubert subvariety. We will explain this fact in detail. 

Let $U\in G(2,V_6)$ be a point in the Grassmannian, that is, $U\subset V_6$ be a $2$-dimensional subspace of $V_6$. Hence, $U^{\perp}=V_6/U$ is a $4$-dimensional quotient of $V_6$. By the perfect pairing $\bigwedge^2 V_6\otimes \bigwedge^4 V_6\to \CC$ we may  interpret $U^{\perp}$ as a linear function on $\bigwedge^2 V_6$, denoted by $H_U$. We compute the hyperplane section $H_U\cap G(V_6,2)$. By definition $H_U: \ker(\bigwedge^2 V_6 \stackrel{\wedge^4 U^{\perp}}{\longrightarrow} \bigwedge^6 V_6 = \CC)$. Thus, 

\begin{align*}
H_U\cap G(V_6,2) & = \{ U'\in G(V_6,2) \ |\ \bigwedge^2 U' \wedge \bigwedge^4 U^{\perp} = 0 \} \\
 & =\{U'\in G(V_6,2) | \dim(U'\cap U^{\perp})\ge 1\}  =: \Sigma_1(U^{\perp})
\end{align*} 
is a Schubert variety. Note that $\dim(U'\cup U^{\perp})\le 5$ for $U'\in H_U\cap G(V_6,2)$, and it is easy to check that 
$$
\Sigma_1(U^{\perp}) = \bigcup_{v\in W} G(U^{\perp}\cup v, 2), 
$$
where $W \oplus U^{\perp}=V_6$. Note that everything is compatible with projectivization. Finally, we see that $\PP^*(H_U)\cap G(V_6,2)\subset \PP^{14}$ is the union of a pencil of Grassmannian of type $G(5,2)$. The converse direction can be shown similarly. 

We conclude that every intersection point of $\PP_*(\ker \lambda)\cap G(2,V_6)$ gives a pencil of elliptic curves on $S$. In order to get $K3$ surfaces with many elliptic pencils of degree $5$, we have to construct a transversal linear section $\PP^8$ such that its dual $\PP_*(\ker \lambda)$ intersects the Grassmannian $G(2,V_6)$ in as many points as possible.

\subsubsection{Extension of elliptic curves to the Grassmannian  $G(V_6,2)$} \label{sec_lift_E}

Let  $(S,L)$ be a Brill--Noether general polarized $K3$ surface of genus $8$ with an elliptic pencil $|E|$ satisfying $L.E = 5$. As $S$ can be embedded (as a linear section) into the Grassmannian $G(V_6,2)$, we will show that every elliptic curve $E'\in |E|$ is a linear section of a sub-Grassmannian  of type $G(5,2)$ of $G(V_6,2)$. 

We need some lemmas. 
We note that $(L-E)^2=4$ and $(L-E) . L =9$, whence $h^0(L-E) \geq 4$ by Serre duality and Riemann--Roch.

\begin{lemma} \label{residualSystem}
 The complete linear system $|L-E|$ is base point free and maps $S$ birationally onto a quartic surface in $\PP^3$ having at most isolated $A_1$-singularities coming from contractions of smooth rational curves $\Gamma$ satisfying $\Gamma . L=\Gamma . E=1$.  
\end{lemma}

 \begin{proof}
 Assume there exists an effective divisor $\Delta$ such that $\Delta^2=-2$ and $\Delta . (L-E) \leq 0$. In particular, $\Delta . E \geq \Delta . L >0$. Then $(L-E-\Delta)^2 \geq 2$, whence $h^0(L-E-\Delta) \geq 3$. As $(S,L)$ is assumed to be Brill--Noether general, we must have $h^0(E+\Delta)=h^0(E)=2$, whence $\Delta . E=1$, and consequently $\Delta . L=1$ and $\Delta . (L-E)=0$. It follows that $L-E$ is nef. It also follows, once we have proved that $|L-E|$ defines a birational morphism, that any connected curve contracted by this morphism is an irreducible rational curve of degree one with respect to $L$ and $E$, proving that the image surface has at most isolated rational $A_1$-singularities.  

 To prove that $|L-E|$ defines a birational morphism, it suffices by the well-known results of Saint-Donat \cite{SD} to prove that there is no irreducible curve $D$ on $S$ satisfying $D^2=0$ and $D . (L-E) =1$ or $2$. If such a $D$ exists, then it is easily seen to satisfy $D . L \geq 5$ by Brill--Noether generality. Hence, $D . E   \geq 3$, so that $(D+E)^2 \geq 6$. It follows that $h^0(D+E) \geq 5$. Since $(L-E-D)^2 \geq 0$ and $(L-E-D) . D \geq 1$, we have $h^0(L-E-D) \geq 2$ by Riemann--Roch and Serre duality, contradicting  Brill--Noether generality.  
 \end{proof}

Let $C\in |L|$ be a smooth curve and let $\cE = \cE_{C,A}$ be the Lazarsfeld--Mukai bundle associated to $C$ and a pencil $|A|$ of  degree $5$  on $C$. Note that the bundle $\cE_{C,A}$ is the unique $L$-stable bundle on $S$ with determinant $L$ and Euler characteristic $6$. 
We write $A_E = E\otimes \OO_C$ and note that $A\ncong A_E$ by Lemma \ref{LMnotstable}. 

\begin{lemma}
  Let $(S,L)$, $E$ and $\cE=\cE_{C,A}$ be as above. Then $h^0(\cE(-E))=1$  and  $h^1(\cE(-E))=h^2(\cE(-E))=0$. In particular, 
$H^0(\cE|_{E})$ is a  five-dimensional  quotient of $H^0(S,\cE)$.
\end{lemma}

\begin{proof}
   Since we know that $h^0(\cE)=6$, the last assertion  immediately follows from the claimed cohomology of $\cE(-E)$ by the obvious restriction sequence.

We will compute the cohomology of $\cE(-E)$ using Serre duality and  the sequence 
\begin{equation} \label{eq:eltrans2}
0\longrightarrow \cE^\vee(E) \longrightarrow H^0(C,A)\otimes \OO_S(E) \longrightarrow A\otimes A_E \longrightarrow 0,
\end{equation}
which is \eqref{eq:eltrans} tensored by $\OO_S(E)$. 

Since $\cE^\vee(E)$ is semi-stable of degree $-4$, one has $h^0(S,\cE^\vee(E))=0$.
Moreover, $h^0(\OO_S(E))=2$ and $h^1(\OO_S(E))=h^2(\OO_S(E))=0$, as $E$ is an irreducible elliptic curve. Hence, the desired cohomology of  $\cE(-E)$ will follow once we prove that 
\begin{equation} \label{eq:cohAA}
h^0(C,A\otimes A_E)=4 \; \; \mbox{and} \; \; h^1(A\otimes A_E)=1. 
  \end{equation}
To prove the latter, note that  $h^0(C,A\otimes A_E)=\chi(H,A\otimes A_E)+h^1(A\otimes A_E)=3+h^1(A\otimes A_E)$ by Riemann--Roch. Since $A \not \cong A_E$, we have $h^0(H,A\otimes A_E) \geq 4$; moreover, equality must hold, as otherwise $h^0(\omega_C \otimes(A\otimes A_E)^{-1})=h^1(A\otimes A_E) \geq 2$ and $\deg (\omega_C \otimes(A\otimes A_E)^{-1})=4$, whence $C$ would contain a $g^1_4$, a contradiction to Brill--Noether generality.  This proves \eqref{eq:cohAA}. 
\end{proof}

Let $E'\in |E|$ be an elliptic curve on $S$. Since $H^0(\cE|_E)$ is a 5-dimensional quotient space of $V_6 = H^0(S,\cE)$, each fiber $\cE_s$ for $s\in E'$ is a 2-dimensional quotient of $H^0(\cE|_E)$ and hence of $V_6$. The image $\phi_{\cE}(E)$ of the elliptic curve is contained in $G(H^0(\cE|_E),2)$. 
Since $\lambda$ is surjective and $E'$ is projectively normal, we have the following commutative diagram 
$$
\begin{xy}
\xymatrix{
	\bigwedge^2 H^0(S,\cE) \ar@{->>}[rr]^{\lambda\ \ \ \ \ \ } \ar@{->>}[d] &&  H^0(S,\bigwedge^2 \cE) \cong H^0(S,L)\ar@{->>}[d]\\
	\bigwedge^2 H^0(E,\cE|_E) \ar[rr] && H^0(E,\bigwedge^2 \cE|_E)\cong H^0(E,L|_E).   
}
\end{xy}
$$
So, we obtain the commutative diagram
$$
\begin{xy}
\xymatrix{
	E' \ar[r]^{\phi_{\cE|_E}\ \ \ \ \ \ } \ar[d]_{\phi_{\bigwedge^2 \cE|_E}} & G(H^0(\cE|_E),2) 
	\ar[d]^{\text{Pl\"ucker}} \ar@{^(->}[r]
	&G(V_6,2) \ar[d]
	 \\
	\PP^4=\PP^*(H^0(E, L|_E)) \ar[r]_\alpha 
	& \PP^*(\bigwedge^2 H^0(E,\cE|_E))\ar@{^(->}[r]
	&\PP^*(\bigwedge^2H^0(S,\cE))
}
\end{xy}
$$
where $\alpha$ is an embedding. 
The diagram is also cartesian. 
Indeed, let $\PP^4=\overline{E'}$ be the linear span, then 
$$ 
E'\subset \PP^4 \cap G(H^0(\cE|_E),2) \subset \PP^4 \cap G(V_6,2) = \PP^4\cap \PP^8 \cap G(V_6,2)= S\cap \PP^4.
$$
But $E'=S\cap \PP^4$ since $|E|$ and $|L-E|$ are base point free (c.f. Lemma \ref{residualSystem}).
Hence, it follows that $E'=\PP^4\cap G(H^0(\cE|_E),2)$. By Section \ref{secDualGrassmannian}, the elliptic pencil $|E|$ on $S$ is cut out by the Schubert cycle $\Sigma_1(V_4)$ on $G(V_6,2)$ for some  four-dimensional  quotient $V_4$. Recall further that there is a one-to-one correspondence between such Schubert cycles and points on the dual Grassmannian $G(2,V_6)$.  

The following corollary follows immediately from our discussion. 

\begin{corollary}\label{pencilsAndDualGrassmannian}
Let $(S,L)$ be a Brill--Noether general polarized $K3$ surface of genus $8$. Let $\PP^5_{(S)}\subset \PP_*(\bigwedge^2H^0(S,\cE))$ be the dual space of $\PP^8=\PP^*H^0(S,L) \subset \PP^*(\bigwedge^2H^0(S,\cE))$. There is a one-to-one correspondence between elliptic pencils $|E|$ on $S$ satisfying $L.E = 5$ and points of $G(2,V_6)\cap \PP^5_{(S)}$. 
\end{corollary}

\subsubsection{Maximal number of distinct elliptic pencils} 
Let $(S,L)$ be a Brill--Noether general $K3$  surface  of genus $8$ and let  $E_1,E_2$  be two classes with  $E_1^2=E_2^2 = 0$   and  $E_1.L= E_2.L = 5$.   Then  $E_1.E_2 = 2$.  
Indeed, the  Hodge Index Theorem  on  $E_1+E_2$  and $L$ yields  $E_1.E_2\le 3$.  Equality implies   $(E_1+E_2)^2=6$  and  $(L-E_1-E_2)^2 = 0$,   whence  $h^0(S, E_1+E_2 )\ge 5$ and $h^0(S,L-E_1-E_2)\ge 2$, a contradiction to Brill--Noether generality. 

On can also see this fact geometrically using the notation of the previous section.
Let $V_5,V_5'$ be two distinct $5$-dimensional subspaces of $V_6$. The intersection of the Grassmannians $G(V_5,2)$ and $G(V_5',2)$ is the Grassmannian $G(V_5\cap V_5',2)$. The Grassmannian $G(V_5\cap V_5',2)$ is a $4$-dimensional quadric. Hence, if $\PP^8$ is a general linear subspace such that its intersection with $G(V_5,2)$ and $G(V_5',2)$ are elliptic curves, then these elliptic curves intersect in two points, namely $\PP^8\cap G(V_5\cap V_5',2)$.

If all our above assumptions are satisfied, we get a $K3$ surface with Picard lattice containing the following lattice 

$$
\begin{pmatrix}
 14 & 5 & 5 & \dots & 5 \\ 
 5 & 0 & 2 & \dots & 2 \\
 5 & 2 & 0  & \ddots & \vdots \\
 \vdots & \vdots & \ddots & \ddots & 2 \\
 5 & 2 & \dots & 2 & 0
\end{pmatrix}.
$$

An easy computation shows that the maximal possible  rank is  $10$ (otherwise  the matrix has at least two positive eigenvalues).  
Let $\fN_9$ be such a lattice of maximal possible rank which is given by the following intersection matrix

$$
\fN_9 = 
\underbrace{
\begin{pmatrix}
 14 & 5 & 5 & \dots & 5 \\ 
 5 & 0 & 2 & \dots & 2 \\
 5 & 2 & 0  & \ddots & \vdots \\
 \vdots & \vdots & \ddots & \ddots & 2 \\
 5 & 2 & \dots & 2 & 0
\end{pmatrix}}_{10 \text{ columns}}.
$$

We denote $S$ a $K3$ surface with the above Picard lattice $\fN_9$ of rank $10$
 (which again exists by \cite[Thm. 2.9(i)]{Mor84} or \cite{Nik80}) 
and let $L$ be the basis element of square $14$, which can be taken to be big and nef by standard arguments (see  \cite[VIII, Prop. 3.10]{BHPV}). 
Let $E_i$,  $i=1,\dots,9$,  be the  generators of square zero. 

\begin{lemma}\label{cr8}
\begin{enumerate}
 \item [(a)] The class $L$ is ample. 
 \item [(b)] The $K3$ surface $(S,L)$ is Brill--Noether general. 
 \item [(c)] The classes $E_1,\dots, E_9$ define elliptic pencils.
\end{enumerate} 
\end{lemma}

This can probably be proved arguing as in the proof of Lemma \ref{cr6}, but the computations are much more tedious. Instead we will give a constructive proof in the next subsection. 

\subsection{ A unirational construction of $K3$ surfaces with nine distinct elliptic pencils}

Recall that any  projective equivalence of two $K3$ surfaces that are linear sections of the Grassmannian $G(2,V_6)$ is induced by an automorphism of $V_6$ (see \cite[Theorem 0.2]{Muk88}).  

By  Corollary \ref{pencilsAndDualGrassmannian}, any Brill--Noether general polarized $K3$ surface $S$ of genus $8$ with exactly nine elliptic pencils of degree five induces and is induced by a unique five-dimensional space $\PP^5_{(S)}$ intersecting $G(2,V_6)\subset \PP^{14}$ in exactly nine points. We reformulate this fact in the following proposition.  To state it  we denote $\cH_{9,5}(G(2,V_6))$ the space of $9$-secant $5$-planes of the Grassmannian $G(2,V_6)\subset \PP^{14}$ intersecting the latter in exactly nine points  and $\widetilde{\cH}_{9,5}(G(2,V_6))$ this space modulo the automorphisms of $V_6$. 

\begin{proposition}\label{oKahn}
 The moduli space of Brill--Noether general polarized $K3$ surfaces of genus $8$ with exactly nine elliptic pencils of degree $5$
 is birational to $\widetilde{\cH}_{9,5}(G(2,V_6))$, and both spaces are non-empty. 
\end{proposition}

\begin{proof}
 By Corollary \ref{pencilsAndDualGrassmannian}, we only need to prove the non-emptiness of $\cH_{9,5}(G(2,V_6))$.  
 A general intersection of $G(2,V_6)$ and a $\PP^7$ is a smooth curve $C$ of genus $8$ and the general curve of genus $8$ is obtained in this way (cf. \cite{Muk93}). Furthermore, a $9$-secant $5$-plane of $G(2,V_6)$ contained in this $\PP^7$ is also a $9$-secant of $C$, which is a divisor in a $g^3_9$ by the geometric Riemann--Roch. Note that the $g^3_9$ is automatically base point free as otherwise the curve would not be Brill--Noether general and thus could not be a linear section of the $G(2,V_6)$ by \cite{Muk93}. Hence
 a general divisor in the $g^3_9$ induces an element of $\cH_{9,5}(G(2,V_6))$.

 We have reduced the problem to constructing a curve of genus $8$ as a linear section of $G(2,V_6)$ carrying a $g^3_9$, or equivalently, taking residuals, a $g^1_5$. Such a curve can be realized as follows: We get a divisor $D$ of degree $5$ in a $g^1_5$ on a curve $C$ of genus $8$ if we fix a $G(2,V_5)$ (where $V_5$ is a $5$-dimensional subspace of $V_6$) and choose a $\PP^7$ such that $C=\PP^7\cap G(2,V_6)$ and $D=\PP^7\cap G(2,V_5)$ induces the $g^1_5=|D|$. In an ancillary file, cf. \cite{M2file}, we have implemented this construction in \emph{Macaulay2} (see \cite{M2}) as well as the construction of the corresponding $K3$ surface. 
\end{proof}

The Picard lattice of the $K3$ surfaces in the moduli space in Proposition \ref{oKahn} contains the lattice $\fN_9$ and the generator of square $14$ is (very) ample and the generators of square $0$ are nef. Let $\cF^{\fN_9}$ be the moduli space of $\fN_9$-lattice polarized $K3$ surfaces. By standard deformation arguments (see \cite[Thm. 14]{Kod64}) the very general element in $\cF^{\fN_9}$ has Picard lattice equal to $\fN_9$, is Brill--Noether general with ample generator of square $14$ and the generators of square $0$ define elliptic pencils.  

\begin{proof}[Proof of Lemma \ref{cr8}]
The last discussion proves the lemma for the very general element in $\cF^{\fN_9}$ having Picard lattice equal to $\fN_9$. Since the properties (a)-(c) of the lemma only depend on the lattice, this finishes the proof. 
\end{proof}

 We also have the following 

\begin{theorem} \label{thm:uniN9}
The moduli space  $\cF^{\fN_9}$ of $\fN_9$-lattice polarized $K3$ surfaces is unirational. 
\end{theorem}

\begin{proof}
The above discussion shows that $\cF^{\fN_9}$ is birational to $\widetilde{\cH}_{9,5}(G(2,V_6))$.  In particular, $\widetilde{\cH}_{9,5}(G(2,V_6))$ is irreducible. 

Consider the following incidence variety 
$$
\{ (V_5^9,\PP^7)\in \cH_{9,5}(G(2,V_6))\times G(8,\Lambda^2V_6)\ |\   V_5^9\subset \PP^7,\ C=\PP^7\cap G(2,V_6) \text{ a smooth curve}\}
$$
and denote $I$ its quotient with the automorphisms of $V_6$ acting diagonally. Then $I$ admits a natural first projection map $\pi_1: I \to \widetilde{\cH}_{9,5}(G(2,V_6))$ and a second projection to the moduli space of curves of genus $8$.  As for $K3$ surfaces, any projective equivalence of two curves of genus $8$ that are linear sections of the Grassmannian $G(2,V_6)$ is induced by an automorphism of $V_6$.  

The proof of Proposition \ref{oKahn} shows that $I$ is non-empty and is therefore birational to a $\PP^3$-bundle over the universal Brill--Noether variety $\cW_{8,9}^3$ by the universal Abel--Jacobi map. Hence $I$ is unirational and irreducible, since $\cW_{8,9}^3\cong \cW_{8,5}^1$ is unirational (and irreducible) by \cite{AC81}.
Since $\pi_1$ is dominant (because $\widetilde{\cH}_{9,5}(G(2,V_6))$ is irreducible), $\widetilde{\cH}_{9,5}(G(2,V_6))$ is unirational. The theorem follows.  
\end{proof}

One may also consider, for $i \in \{0,\ldots,8\}$, the moduli spaces $\cF^{\fN_i}$ of $\fN_i$-lattice polarized $K3$ surfaces, where $\fN_i$ is the rank $i+1$ lattice 
$$
\fN_i = 
\underbrace{
\begin{pmatrix}
 14 & 5 & 5 & \dots & 5 \\ 
 5 & 0 & 2 & \dots & 2 \\
 5 & 2 & 0  & \ddots & \vdots \\
 \vdots & \vdots & \ddots & \ddots & 2 \\
 5 & 2 & \dots & 2 & 0
\end{pmatrix}}_{i+1 \text{ columns}}
$$
Then $\dim \cF^{\fN_i}=19-i$ and $\cF^{\fN_{i+1}} \subset \cF^{\fN_i}$ for each $i \in \{0,\ldots,8\}$. Note that $\cF^{\fN_0} = \cF_8$. 

\begin{theorem} \label{thm:uniN6}
 The moduli spaces $\cF^{\fN_i}$ of $\fN_i$-lattice polarized $K3$ surfaces are unirational for $i\le 6$. 
\end{theorem}

\begin{proof}
The case $i=0$ is proved in \cite{Muk88}. 
By Corollary \ref{pencilsAndDualGrassmannian} and Lemma \ref{cr8}, the general $K3$ surface in $\cF^{\fN_i}$ corresponds uniquely to a five-dimensional projective space intersecting the Grassmannian $G(2,V_6)\subset \PP^{14}$ in exactly $i$ points modulo automorphisms of $V_6$. Such $i$-secant $5$-planes are unirationally parametrized by the product of the $i$-th symmetric product of $G(2,V_6)$ and $(6-i)$-th symmetric product of $\PP^{14}$. 
\end{proof}

We remark that the unirationality of $\cF^{\fN_1}$ can also be shown using quartic surfaces in $\PP^3$ containing an elliptic quintic curve. 
The question of (uni)rationality of $\cF^{\fN_7}$ and $\cF^{\fN_8}$ is open.

\subsection{The moduli map}

Let $\cF_8$ denote the $19$-dimensional moduli space of polarized $K3$ surface of genus $8$ and 
$\cP_8$ the moduli space of triples $(S,L,C)$ where $(S,L) \in \cF_8$ and $C\in |L|$ is a smooth irreducible curve. Let $ m_8: \cP_8 \longrightarrow \cM_8$ 
be the moduli map. 

\begin{proposition} \label{prop:6dimfiber}
  Let $(S,L) \in \cF_8$ be a Brill--Noether general $K3$ surface such that $S$ contains an elliptic pencil $|E|$ satisfying $E . L=5$. Then the fiber of $m_8$ is smooth and  $6$-dimensional at any point represented by a smooth curve $C$ in $|L|$.
\end{proposition}

\begin{proof}
By comparing dimensions, the fibers of $m_8$ are at least $6$-dimensional. (It is known that  $m_8$   is dominant,  whence its general fibers are precisely $6$-dimensional,  but we will not use this.) By \cite[\S 3.4.4]{Ser} or \cite{beau}, the kernel of the differential of $m_8$ at a point $(S,L,C)$ is isomorphic to $H^1(\T_S(-L))$.  To prove the proposition, it therefore suffices  by Serre duality  to prove that  $h^1(\Omega_S(L)) \leq 6$.
 
Let $\varphi: S \to \PP^3$ be the morphism defined by $|L-E|$ and $S_0$ be its image, which is a quartic surface. By Lemma \ref{residualSystem} its possible singularities are images of contracted disjoint rational curves $\Gamma_i$ on $S$, $i=1,\ldots,k$.
By \cite[Thm. 2.1]{mor} we have a short exact sequence
\begin{equation}
  \label{eq:sesmor}
  \xymatrix{
0 \ar[r] & \OO_{\Gamma_1 + \cdots + \Gamma_k}  \ar[r] & \varphi^*\Omega_{S_0}  \ar[r] & 
\Omega_S \ar[r] & \OO_{\Gamma_1 + \cdots + \Gamma_k}  \ar[r] & 0.
}
\end{equation}
Twisting by $\OO_S(L)$, taking cohomology and using the fact that $\Gamma_i \cdot L=1$ by Lemma \ref{residualSystem}, we obtain
\begin{equation}
  \label{eq:a1}
  h^1(\Omega_S(L)) \leq h^1(\varphi^*\Omega_{S_0}(L)).
\end{equation}

Pulling back the conormal bundle sequence
\[
    \xymatrix{   & \OO_{S_0}(-4) \cong \cI_{S_0/\PP^3}/\cI_{S_0/\PP^3}^2 \ar[r] & \Omega_{\PP^3}|{_{S_0}} \ar[r]  & \Omega_{S_0} \ar[r] & 0}
\]
and twisting by $\OO_S(L)$, we obtain
\[
    \xymatrix{  & \OO_{S}(-3L+4E) \ar[r] & \varphi^*\Omega_{\PP^3}|{_{S_0}}(L) \ar[r]  & \varphi^*\Omega_{S_0}(L) \ar[r] & 0.}
\]
The left hand map is injective, as $ \OO_{S}(-3L+4E)$ is locally free. Thus,
\begin{equation}
  \label{eq:a2}
  h^1(\varphi^*\Omega_{S_0}(L)) \leq h^1(\varphi^*\Omega_{\PP^3}|{_{S_0}}(L))+h^0(3L-4E),
\end{equation}
using Serre duality. 
Pulling back the dual of the Euler sequence,
\[
\xymatrix{ 
0 \ar[r] & \Omega_{\PP^3}|{_{S_0}} \ar[r]  &  H^0(\OO_{S_0}(1)) \otimes \OO_{S_0}(-1)  \ar[r]  & \OO_{S_0} \ar[r]  & 0
}
\]
and twisting by $\OO_S(L)$, we obtain
\[
\xymatrix{ 
0 \ar[r] & \varphi^*\Omega_{\PP^3}|{_{S_0}}(L) \ar[r]  & H^0(L-E) \otimes \OO_{S}(E)  \ar[r]  & \OO_{S}(L) \ar[r]  & 0.
}
\]
Hence, since $h^1(E)=0$ as $E$ is irreducible, we obtain
\begin{equation}
  \label{eq:a3}
  h^1(\varphi^*\Omega_{\PP^3}|{_{S_0}}(L)) \leq \cork \mu,
\end{equation}
where $\mu$ is the multiplication map of sections
\[ \mu:H^0(L-E) \otimes H^0(E)  \longrightarrow H^0(L).\]

Combining \eqref{eq:a1}, \eqref{eq:a2} and \eqref{eq:a3}, we see that we obtain the desired inequality \linebreak $h^1(\Omega_S(L)) \leq 6$ if we prove that 
\begin{equation} \label{eq:a4}
  h^0(3L-4E)=5
\end{equation}
and
\begin{equation} \label{eq:a5}
  \cork \mu =1.
\end{equation}
 
To prove \eqref{eq:a5}, note that the evaluation map $H^0(E) \otimes \OO_S \to \OO_S(E)$ is surjective as $|E|$ is base point free and has kernel $\OO_S(-E)$. Twisting by $\OO_S(L-E)$, we obtain
\[
\xymatrix{ 
0 \ar[r] & \OO_S(L-2E) \ar[r]  &  H^0(E) \otimes \OO_S(L-E)  \ar[r]  & \OO_{S}(L) \ar[r]  & 0
}
\]
Taking cohomology and using the fact that $h^1(L-E)=0$ as $L-E$ is big and nef by Lemma \ref{residualSystem}, we obtain that $\cork \mu=h^1(L-2E)$. 

We have $(L-2E) . L=4$, whence $h^2(L-2E)=h^0(2E-L)=0$, as $L$ is ample.
Similarly, $h^0(L-2E)=0$, since $(L-2E) . (L-E)=-1$ and $L-E$ is nef. 
Since $(L-2E)^2=-6$, Riemann--Roch yields $h^1(L-2E)=1$, and \eqref{eq:a5} is proved.

To prove \eqref{eq:a4}, note that $(3L-4E)^2=6$ and $h^2(3L-4E)=h^0(4E-3L)=0$, as $(4E-3L) . E <0$ and $E$ is nef. Hence, \eqref{eq:a4} is equivalent to 
$h^1(3L-4E)=0$. 

To get a contradiction, assume that $h^1(3L-4E)>0$. Then, by \cite{klvan}, there exists an effective divisor $\Delta$ such that $\Delta^2=-2$ and $k:=-\Delta . (3L-4E) \geq 2$. Since $\Delta . L >0$, as $L$ is ample, we must have 
\begin{equation}
  \label{eq:2onE}
  \Delta . E \geq 2.
\end{equation}

One computes $(3L-4E-k\Delta)^2=6$ and
$(3L-4E-k\Delta) . (L-E)=7-k\Delta . (L-E)$. By the Hodge index theorem,
\[ 24 = (3L-4E-k\Delta)^2 \cdot (L-E)^2 \leq \left[7-k\Delta . (L-E) \right] ^2,\]
whence the only possibilities
\begin{itemize}
\item[(I)] $\Delta . (L-E)=0$; or
\item[(II)] $\Delta . (L-E)=1$ and $k=2$.
\end{itemize}

In case (I) we find $(L-E-\Delta)^2=2$ and $(L-E-\Delta) . (L-E)=4$, whence $h^0(L-E-\Delta) \geq 3$ by Riemann--Roch and Serre duality. By \eqref{eq:2onE} we have $(E+\Delta)^2 \geq 2$, whence also $h^0(E+\Delta) \geq 3$ by Riemann--Roch. But then $h^0(L-E-\Delta)h^0(E+\Delta) \geq 9 = 8+1$, contradicting Brill--Noether generality. 

In case (II) we have $\Delta . L = \Delta . E +1$ and $-2=\Delta . (3L-4E)$, which together yield $\Delta . E=5$ and $\Delta . L=6$. Therefore, $(L-E-\Delta)^2=0$ and $(L-E-\Delta) . L=3$, whence
 $h^0(L-E-\Delta) \geq 2$ by Riemann--Roch and Serre duality. Moreover,
$(E+\Delta)^2 =8$, whence $h^0(E+\Delta) \geq 6$ by Riemann--Roch. Similarly, to the previous case, we obtain a contradiction to  Brill--Noether generality. 

This shows that \eqref{eq:a4} holds and finishes the proof of the proposition.
\end{proof}

For $i\in \{0,\dots,9\}$, let $\fN_i$ and $\cF^{\fN_i}$ be as in the previous subsection and  

let $\cP^{\fN_i}$ be the moduli space of triples as in the introduction. Note that $\cP^{\fN_i}$ is birational to the open part of the tautological $\PP^8$-bundle over $\cF^{\fN_i}$ consisting of pairs $(S,C)$ with $[S] \in \cF^{\fN_i}$ and $[C]$ representing a smooth curve in  $|L|$, where $L$ is the generator class of square $14$ in $\fN_i$. We have $\cP^{\fN_{i+1}} \subset \cP^{\fN_i}$ for each $i \in \{0,\ldots,8\}$.

Let $m^{\fN_i}_8:\cP^{\fN_i} \to \cM_8$ be the moduli map. 

\begin{proposition} \label{prop:expfiber}
  For each $i \in \{0,\ldots,9\}$, a general fiber of $m^{\fN_i}_8$ has dimension \linebreak 
$\max \{0,6-i\}$.  
\end{proposition}

\begin{proof}
By Proposition \ref{prop:6dimfiber}, the fiber of $m^{\fN_0}_8$ is smooth and $6$-dimensional at any point $(S,C) \in \cP^{\fN_9}$. Fix such an $(S,C)$.

We will show that there exists a chain of irreducible components $F_i\subset (m_8^{\fN_i})^{-1}([C])$ of the fiber of $m_8^{\fN_i}$ for $i\in \{0,\dots, 5\}$, respectively, containing $(S,C)\in \cP^{\fN_9}$ such that 
$$
(S,C) \in F_5 \subsetneq F_4 \subsetneq \dots \subsetneq F_1 \subsetneq F_0.
$$
Consequently, there exist $K3$ surfaces $S_i\in \cF^{\fN_i}\backslash \cF^{\fN_{i+1}}$ for $i\in \{0,\dots, 5\}$ containing $C$. 
Since $\dim F_0 = 6$ by Proposition \ref{prop:6dimfiber}, the dimension of $F_i$ is $6-i$ for $i\in \{0,\dots, 5\}$ and the proposition will follow.

By construction, $S$ (resp. $C$) is the intersection of $G(V_6,2)$ with a $\PP^8$ (respectively a $\PP^7$) in $\PP^{14}$. The dual $\PP^5$ of the $\PP^8$, which we henceforth call $\PP^5_{(S)}$, intersects the dual $G(2,V_6)$ in $9$ points, call them $x_1,\ldots,x_9$,   
and the dual $\PP^6$ of the $\PP^7$, which we henceforth call $\PP^6_{(C)}$, contains 
$\PP^5_{(S)}$. 
 
By construction, the nine points $x_1,\ldots, x_9$ span $\PP^5_{(S)}$. Thus, we may find inside $\PP^6_{(C)}$ a set of six additional hyperplanes 
$\PP^5_{(i)}$, $i \in \{0,\ldots,5\}$ containing precisely $i$ of the points
$x_1,\ldots,x_9$; in particular $\PP^5_{(i)}$ intersects $G(2,V_6)$ in precisely $i$ points.

Denote by $\PP^8_{(i)}$ the dual $\PP^8$ of $\PP^5_{(i)}$. Then $\PP^8_{(i)} \cap G(V_6,2)$ is a $K3$ surface $S_i$ containing $C$ and precisely $i$ elliptic pencils of degree $5$ (and mutually intersecting in $2$ points)  by Corollary \ref{pencilsAndDualGrassmannian}.  As the nine elliptic pencils together with $C$ generate $\fN_9\subset \Pic(S)$, we also have that $C$ and the  
$i$ elliptic pencils generate $\fN_i\subset \Pic(S_i)$, whence 
$S_i \in \cF^{\fN_{i}}\backslash\cF^{\fN_{i+1}}$. Each pair $(S_i,C)$ therefore lies in $F_i \backslash F_{i+1}$.  This concludes the proof.
\end{proof}

\begin{corollary}\label{dybala8}
 For each $i \in \{0,\ldots,9\}$, the codimension of the image of the moduli map $m^{\fN_i}_8$ is $\max \{0,i-6\}$. In particular, a general curve of genus $8$ is a linear section of a $K3$ surface such that precisely six out of its $14$ $g^1_5$s are induced by elliptic pencils on the $K3$ surface. Moreover, there is a codimension $k$ family of curves lying on a $K3$ surface such that precisely $6+k$ of its $g^1_5$s are induced by elliptic pencils on the $K3$ surface for $k\in \{1,2,3\}$. 
\end{corollary}

\begin{remark}
 One can ask similar questions for $K3$ surfaces of higher even genus. For instance, how many elliptic pencils of minimal degree  exist  on a Brill--Noether general $K3$ surface? But the methods in this article cannot be applied to $K3$ surfaces of higher genus. Indeed, let $C$ be a Brill--Noether general curve of even genus $g\ge 10$. Note on the one hand that the curve $C$ does not lie on a $K3$ surface and on the other hand that the (finite) number of pencils of minimal degree on $C$ is bigger that $19$ (the maximal rank of the Picard lattice of a smooth $K3$ surface). Furthermore, a characterization of Brill--Noether general $K3$ surfaces is only known for $g\le 10$ and $12$.
\end{remark}

\end{document}